\documentclass{article}
\usepackage[utf8]{inputenc}
\usepackage{hyperref}
\usepackage[numbers]{natbib}
\usepackage{indentfirst}
\usepackage[printwatermark]{xwatermark}
\usepackage{xcolor}
\usepackage{graphicx}
\usepackage{lipsum}
\usepackage{amssymb}
\usepackage{amsfonts}
\usepackage{amsmath}
\usepackage{amssymb}
\usepackage{amsthm}

\usepackage[a4paper, total={6in, 10in}]{geometry}

\theoremstyle{plain} 
\newtheorem{theorem}{Theorem}[section]

\newtheorem{lemma}[theorem]{Lemma}
\newtheorem{proposition}[theorem]{Proposition}

\theoremstyle{definition}
\newtheorem{definition}{Definition}[theorem]
\newtheorem{remark}{Remark}[theorem]
\usepackage[capitalize]{cleveref}
	\crefname{claim}{Claim}{Claims}
	\Crefname{claim}{Claim}{Claims}
	\crefname{app-corollary}{Corollary}{Corollaries}
	\Crefname{app-corollary}{Corollary}{Corollaries}
	\crefname{app-definition}{Definition}{Definitions}
	\Crefname{app-definition}{Definition}{Definitions}
	\crefname{figure}{Figure}{Figures}
	\Crefname{figure}{Figure}{Figures}
	\crefname{lemma}{Lemma}{Lemmata}
	\Crefname{lemma}{Lemma}{Lemmata}
	\crefname{app-lemma}{Lemma}{Lemmata}
	\Crefname{app-lemma}{Lemma}{Lemmata}
	\crefname{app-proposition}{Proposition}{Proposition}
	\Crefname{app-proposition}{Proposition}{Proposition}
	\crefname{app-theorem}{Theorem}{Theorems}
	\Crefname{app-theorem}{Theorem}{Theorems}

\title{Topological uniqueness results for Lefschetz fibrations over the disc}
\author{A.~A.~Kazhymurat}
\date{August 2018}

\begin{document}

\maketitle
\begin{abstract}
    We prove that a Lefschetz fibration over the disc that, after compactification, has the same singular fibers as an extremal rational elliptic surface can be obtained by deleting a singular fiber and a section from the rational extremal elliptic surface, i.e. such a Lefschetz fibration is determined up to topological equivalence by its set of singular fibers. 
    We get a complete clasification of Lefschetz fibrations with 2 $I_1$ fibers as a byproduct of our results. 

    The proof is inspired by homological mirror symmetry and Karpov--Nogin's theorem on constructivity of helices on del Pezzo surfaces. 

    It would be interesting to extend our results to the case of Lefschetz fibrations that, after compactification, have the same singular fibers as an extremal elliptic K3 surface. 
\end{abstract}
\section{Introduction}
Morse theory describes the topology of a smooth manifold in terms of the critical points of a generic smooth real-valued function on it. Given a Morse function on a smooth manifold one can construct a handle decomposition of this manifold (see e.g. \cite{Milnor}). 

Picard--Lefschetz theory is a complex analogue of Morse theory where the role of Morse functions is played by Lefschetz fibrations, i.e. holomorphic maps to the Riemann sphere with non-degenerate critical points.

Lefschetz fibrations were later extended to the setting of symplectic manifolds. Donaldson has shown that every compact symplectic manifold, possibly after removing a real codimension 2 submanifold, is a total space of some Lefschetz fibration. Because the class of compact symplectic manifold is topologically diverse (for example, every finitely presented group arises as a fundamental group of some compact symplectic 4-manifold \cite{Gompf1}), this shows that Lefschetz fibrations are a powerful instrument to understand the topology of 4-manifolds. 

There is no general classification result known for Lefschetz fibrations. The rich topological behaviour displayed by Lefschetz fibrations makes such a classification inaccessible with the present tools; for example, Lefschetz fibrations of genus at least 2 over the disk were used by Ozbagci \cite{Ozbagci} to find contact 3-manifolds admitting infinitely many pairwise non-diffeomorphic Stein fillings. Prior to Auroux \cite{Auroux}, however, it was generally believed that the classification of Lefschetz fibrations of genus 1 over the disk is comparatively simple. 

Auroux \cite{Auroux} constructed examples of 2 inequivalent Lefschetz fibrations of genus 1 over the disk with 2 singular fibers. His construction relied on the connection between Lefschetz fibrations and factorizations in the mapping class group. His results imply that there exist contact 3-manifolds admitting 2 inequivalent Stein fillings with diffeomorphic total spaces. This demonstrates the interesting topology of Lefschetz fibrations with 2 singular fibers over the disc. 

The main result of the paper is that for each of the 14 extremal rational types there exists only one Lefschetz fibration up to topological equivalence. This implies, in particular, that every Lefschetz fibration of extremal rational type is algebraic. Furthermore, we obtain a complete classification of Lefschetz fibrations with 2 singular fibers of type $I_1$. 

In Section~\ref{Lefschetz-section} we define Lefschetz fibrations and describe their relationship to monodromy factorizations in the mapping class group. In Section~\ref{section-MCG} we prove Proposition~\ref{intersections} connecting monodromy factorizations and algebraic intersection numbers of vanishing cycles. In Section~\ref{section-1119} we use Proposition~\ref{intersections} to prove that the algebraic intersection numbers of vanishing cycles satisfy Markov type equations. The transitivity of the braid group action on the set of positive integral solutions of Markov type equations, first established by Karpov--Nogin \cite{Karpov}, then implies the topological uniqueness of the Lefschetz fibration over the disc of extremal rational type. 


 
\section{Preliminaries}
\subsection{Mapping class group}
Let $\Sigma$ be a compact oriented surface with boundary $\partial \Sigma$. \textit{The mapping class group} of $\Sigma$ is the group $\mathrm{MCG}(\Sigma)$ of isotopy classes of orientation-preserving homeomorphisms of $\Sigma$ that fix $\partial M$ pointwise.

Let $\Sigma_{1, 1}$ be the torus with one boundary component. It is known that $\mathrm{MCG}(\Sigma_{1, 1})\approx B_3$ \cite{Margalit}.  

The first singular homology group of $\Sigma_{1,1}$ has rank 2 as a $\mathbb{Z}$-lattice. The intersection form $\langle \cdot, \cdot \rangle$ defines a symplectic form on $H_1(\Sigma_{1, 1},\mathbb{Z})$. 

The group $\mathrm{MCG}(\Sigma_{1, 1})$ has a natural representation on $H_1(\Sigma_{1, 1},\mathbb{Z})$ called the \textit{symplectic representation}. This representation is a surjective homomorphism $\mathrm{MCG}(\Sigma_{1, 1})\rightarrow SL(2, \mathbb{Z})$ whose kernel coincides with the center of $\mathrm{MCG}(\Sigma_{1, 1})$ (see for example \cite{Margalit}). A Dehn twist around a simple closed curve $C$ acts on a homology class $\gamma \in H_1(\Sigma_{1, 1}, \mathbb{Z})$ as follows
\begin{equation}
\label{dehneffect}
\tau_{C}\gamma=\gamma+\langle [C], \gamma\rangle [C].    
\end{equation}

The group $\mathrm{MCG}(\Sigma_{1, 1})$ is generated by the Dehn twists $\tau_a$, $\tau_b$ around any two simple closed curves intersecting transversely at one point. The Dehn twist around a curve parallel to the boundary component is $\delta=(\tau_a \tau_b)^6$. It generates the kernel of the symplectic representation $\mathrm{MCG}(\Sigma_{1, 1})\rightarrow SL(2, \mathbb{Z})$.

The abelianization of $\mathrm{MCG}(\Sigma_{1, 1})$ is $\mathbb{Z}$. Under the abelianization map Dehn twists around non-separating simple closed curves map to $1$, while $\delta$ maps to $12$ \cite{Margalit}. 
\begin{remark}
For $\Sigma_{1, 1}$, every primitive homology class contains exactly one isotopy class of simple closed curves \cite{Rivin}, \cite{Margalit}. This means that to specify the Dehn twist $\tau_C$ around a curve $C$ it is enough to give the homology class $[C]\in H_1(\Sigma_{1, 1}, \mathbb{Z})$. 
\end{remark}
\label{Lefschetz-section}
\subsection{Lefschetz fibrations}
A \textit{Lefschetz fibration} over the disc is a smooth surjective map $f:M\rightarrow D^2$ from a compact oriented smooth 4-dimensional manifold $M$ with boundary $\partial M$ to the closed 2-disc $D^2$ having the following properties:
\begin{itemize}
    \item it is a submersion away from the critical points; 
    \item it has finitely many critical points and each critical point is non-degenerate;
    \item all critical values lie in the interior of $D^2$;
    \item each critical point has an orientation-preserving complex chart in which $f(z_1, z_2)=z_1^2+z_2^2$.
\end{itemize}
Note that some authors require $f$ to be injective on the set of critical points. 

Let $f:M\rightarrow D^2$ be a Lefschetz fibration with critical points $p_1,\dots, p_r$. A \textit{smooth fiber} of the Lefschetz fibration $f:M\rightarrow D^2$ is the preimage under $f$ of a point in $D^2\backslash \{f(p_1), \dots, f(p_r)\}$. Ehresmann's lemma shows that smooth fibers corresponding to two different points in  $D^2/\{f(p_1), \dots, f(p_r)\}$ are diffeomorphic. 

Analogously, a \textit{singular fiber} of the Lefschetz fibration $f:M\rightarrow D^2$ is the preimage under $f$ of $f(p_i)\in D^2$ for some $i$, $1 \leq i \leq r$. The Lefschetz fibrations we consider have singular fibers of the following type:
\begin{itemize}
    \item type $I_1$, i.e. an immersed sphere of homological self-intersection 0 with one positive double point;
    \item type $I_n$, $n\geq 2$, i.e. a chain of $n$ spheres with homological self-intersection $-2$.
\end{itemize}


Let $f:M\rightarrow D^2$ be a Lefschetz fibration with critical values $p_1, \dots, p_r$. Ehresmann's lemma implies that the restriction of $f$ to a map $M\backslash\{f^{-1}(p_1), \dots, f^{-1}(p_r)\}\rightarrow D^2\backslash \{p_1, \dots, p_r\}$ is a fiber bundle $E$. Fix a reference point $p_*\in D^2\backslash\{p_1,\dots, p_r\}$ whose fiber $f^{-1}(p_*)$ we denote as $F$. Choose a set of paths $l_i:[0, 1]\rightarrow M\backslash\{f^{-1}(p_1), \dots, f^{-1}(p_r)\}$ for $1 \leq i \leq r$ such that $l_i(0)=l(i)=p_*$ and such that $l_i$ encloses $p_i$ and no other critical values. The paths $l_i$ are called \textit{vanishing paths}. We can consider the pullback of the fiber bundle $E$ to $[0, 1]$. Since every fiber bundle over $[0, 1]$ is trivial, there exists a trivialization $T:[0, 1]\times F\rightarrow l_i^*E$. The induced diffeomorphism $T|_{\{1\}\times F}\cdot T^{-1}|_{\{0\}\times F}$, which is a well-defined element of the mapping class group of $F$, is called \textit{the monodromy} around $p_i$.  It can be shown that the monodromy around a critical value $p_i$ corresponding to a fiber of type $I_n$ is equal to the $n$-th power of a Dehn twist around some simple closed curve $C \subset F$; this curve is called \textit{the vanishing cycle} of the fiber $f^{-1}(p_i)$ \cite{Gompf}. 

The ordered set of monodromies around points $p_i$ is called \textit{the monodromy factorization} of $f$. Lefschetz fibration can be reconstructed (up to topological equivalence) from its monodromy factorization. There are two group actions on the set of monodromy factorizations that preserve the corresponding Lefschetz fibration:
\begin{itemize}
    \item The braid group 
    \begin{equation}
    \label{artinpresentation}
    B_r=\langle \sigma_1, \dots, \sigma_{r-1}|\sigma_i \sigma_j=\sigma_j\sigma_i \: \mathrm{if}\:|i-j|\geq 2, \: \sigma_i\sigma_{i+1}\sigma_i=\sigma_{i+1}\sigma_i\sigma_{i+1} \: \mathrm{for}\:  1\leq i \leq n-2\rangle
    \end{equation} acts by \textit{Hurwitz moves} (or \textit{mutations})
$$
\sigma_i:(\tau_1, \dots, \tau_i, \tau_{i+1}, \dots, \tau_r)\rightarrow (\tau_1, \dots, \tau_i \tau_{i+1} \tau_i^{-1}, \tau_i, \dots, \tau_r) \qquad \mathrm{for}\:1 \leq i \leq r. 
$$
\item The mapping class group of the smooth fiber $F$ acts by \textit{global conjugation}
$$
\phi:(\tau_1, \dots, \tau_r)\rightarrow (\phi \tau_1 \phi^{-1}, \dots, \phi \tau_r \phi^{-1}). 
$$
\end{itemize}
The problem of classifying Lefschetz fibrations over the disc is equivalent to the problem of classifying the equivalence classes of monodromy factorizations up to Hurwitz moves and global conjugation (see for example \cite{AurouxMap}).

\subsection{Lefschetz fibrations of genus 1}
We restrict to the case of Leschetz fibrations with smooth fiber diffeomorphic to a torus with one boundary component $\Sigma_{1, 1}$. 
\begin{definition}
\label{extremal}
\textit{A Lefschetz fibration of extremal rational type} is a Lefschetz fibration having 3 singular fibers of type $I_{l_0}$, $I_{m_0}$, $I_{n_0}$ for $l_0$, $m_0$, $n_0$ assuming values indicated in Table~1 such that, for some choice of vanishing paths, the vanishing cycles $C_1$, $C_2$, $C_3$ are non-separating curves and satisfy
\begin{equation}
\label{factorization9111}
\tau_{C_1}^{l_0}\tau_{C_2}^{m_0}\tau_{C_3}^{n_0}=\delta \tau_{C}^{m_0+l_0+n_0-12}.
\end{equation}
for some non-separating simple closed curve $C \subset \Sigma_{1, 1}$.
\end{definition}

\begin{table}[h!]
\label{delpezzo}
\caption{The possible values of $l_0$, $m_0$, $n_0$, and the configuration of vanishing cycles with the least possible intersection numbers realizing the given extremal rational type. Here $[u]$, $[v]$ denote some symplectic basis of $H_1(\Sigma_{1, 1}, \mathbb{Z})$.  }
\begin{tabular}{ |p{2cm}|p{1cm}|p{1cm}|p{1cm}|p{2cm}|p{2cm}|p{2cm}|p{1cm} | }
 
 \hline
 Number & $l_0$ &$m_0$ &$n_0$ &$[C_{1, min}]$ & $[C_{2, min}]$ &$[C_{3, min}]$&$[C_{min}]$\\
 \hline
 (1)   & 1    &1&   1 & $[v]-3[u]$ & $[v]$ & $[v]+3[u]$&$[u]$\\
 (2)&   1  & 1   &2& $[v]-4[u]$&$[v]$&$[v]+2[u]$&$[u]$ \\
 (3) & 1 & 2&  3& $[v]-3[u]$& $[v]$&$[v]+[u]$&$[u]$\\
 (4)    &1 & 1&  5 & $[v]+3[u]$ & $2[v]+[u]$&$[v]$&$[u]$\\
 (5)&   2  & 2&4 &$[v]-2[u]$& $[v]$ & $[v]+[u]$&$[u]$\\
 (6)& 3  & 3   &3 & $[v]-3[u]$ & $[v]$ & $[v]+3[u]$&$[u]$\\
 (7)& 1  & 2&6 & $2[v]-3[u]$& $[v]$& $[v]+[u]$&$[u]$\\
 (8)   & 1    &1&  8 &$2[v]-3[u]$&$2[v]-[u]$&$[v]$&$[u]$\\
 (9)&   2  & 4   &4 & $2[v]-[u]$&$[v]$&$[v]$&$[u]$\\
 (10) &1 & 3&  6 & $3[v]-2[u]$&$[v]$&$[v]+[u]$&$[u]$\\
 (11)    &1 & 1&  9& $3[v]-2[u]$& $3[v]-[u]$ &$[v]$&$[u]$\\
 (12)&   2  & 2 &8& $4[v]-3[u]$&$2[v]-[u]$&$[v]$&$[u]$\\
 (13)& 2 & 3   &6 & $3[v]-2[u]$& $2[v]-[u]$&$[v]$&$[u]$\\
 (14)& 1  & 5 &5& $5[v]-3[u]$&$2[v]-[u]$&$[v]$&$[u]$\\
 \hline
\end{tabular}
\end{table} 
\begin{remark}
Each of the 14 extremal rational types is realized by a Lefschetz fibration constructed by deleting a singular fiber and a section from an extremal rational elliptic surface. Note that extremal rational elliptic surfaces have been completely classified by Miranda--Ulf~\cite{Miranda1986}.  
\end{remark}

\section{Computations in the mapping class group}
\label{section-MCG}
Binary quadratic forms $f$ and $g$ are called \textit{equivalent} if there exist $a$, $b$, $c$, $d\in \mathbb{Z}$ such that
$$
f(a x+by,cx+dy)=g(x, y), \qquad ad-bc=1. 
$$
It is well-known that the discriminants of equivalent quadratic forms are equal. 

Let us fix a symplectic basis $[u]$, $[v]$ of $H_1(\Sigma_{1, 1}, \mathbb{Z})$. For any $\phi \in \mathrm{MCG}(\Sigma_{1,1})$, the expression
$$
\langle \phi \gamma, \gamma \rangle 
$$
for $\gamma=p[v]+q[u]$ defines a binary quadratic form in $p$, $q$. We denote its discriminant as $d(\phi)$. Note that $d(\phi)$ does not depend on the choice of symplectic basis of $H_1(\Sigma_{1, 1}, \mathbb{Z})$ because any two symplectic bases are related by an element of $SL(2, \mathbb{Z})$ (so the resulting quadratic forms are equivalent). 
\begin{lemma}
\label{discriminant}
Let $C_1$, $C_2\subset \Sigma_{1, 1}$ be simple closed curves with intersection $[C_1]\cdot [C_2]=a$. Then for positive integers $m$, $n$ we have
$$
d(\tau^m_{C_1}\tau^n_{C_2})=m^2 n^2 a^4-4mn a^2.
$$ 
\end{lemma}
\begin{proof}
Choose a symplectic basis $[u]$, $[v]$ of $H_1(\Sigma_{1, 1}, \mathbb{Z})$ such that 
$$
[C_1]=[u], \qquad [C_2]=a[v]+b[u] 
$$
for some $b\in \mathbb{Z}$. 

Compute the value of the quadratic form $\langle \tau_{C_1}^m\tau^n_{C_2}\gamma, \gamma \rangle$ for a homology class $\gamma=p[v]+q[u]$
\begin{align*}
\tau^n_{C_2}\gamma-\gamma&=n\langle [C_2], \gamma \rangle [C_2], \\
\tau^m_{C_1}\tau^n_{C_2}\gamma-\gamma&=m\langle [C_1], \gamma \rangle [C_1]+n\langle [C_2], \gamma \rangle [C_2]+mn\langle [C_2], \gamma \rangle \langle [C_1], [C_2]\rangle [C_1],\\
\langle \tau^m_{C_1}\tau^n_{C_2}\gamma, \gamma \rangle&=m\langle [C_1], \gamma \rangle^2+n\langle [C_2], \gamma\rangle^2+mn\langle [C_1], \gamma \rangle \langle [C_2], \gamma \rangle \langle [C_1], [C_2] \rangle=\\
mp^2+n(bp-aq)^2+mnp(bp-aq)a&=(m + n b^2 + mnab)p^2-a(2 nb + mna)pq+na^2 q^2.
\end{align*}
Therefore, 
$$
d(\tau^m_{C_1}\tau^n_{C_2})=a^2(2nb + mna)^2-4(m + n b^2 + mnab)na^2=m^2 n^2 a^4-4mn a^2.
$$
\end{proof}
\begin{proposition}
\label{intersections}
Let $C_1$, $C_2$, $C_3$, $C_4\subset \Sigma_{1, 1}$ be non-separating simple closed curves such that
\begin{equation}
\label{klmn}
\tau_{C_1}^m\tau_{C_2}^n=\delta \tau_{C_3}^{-k}\tau_{C_4}^{-l}
\end{equation}
for some $m, n, k, l \in \mathbb{N}$. Then  
$$
mn\langle [C_1], [C_2]\rangle^2=lk\langle [C_3], [C_4] \rangle^2.
$$
\end{proposition}
\begin{proof}
Note that under the abelianization map $\mathrm{MCG}(\Sigma_{1, 1})\rightarrow \mathbb{Z}$ the element $\delta$ maps to $12$ while Dehn twists around non-separating curves map to $1$. Therefore, Equation \eqref{klmn} implies that $m+n+k+l=12$. 

Let us consider the action of both sides of Equation \eqref{klmn} on a homology class $\gamma \in H_1(\Sigma_{1, 1}, \mathbb{Z})$
$$
\gamma+n\langle [C_2], \gamma \rangle [C_2]+m\langle [C_1], \gamma \rangle [C_1]+mn\langle [C_2], \gamma \rangle \langle [C_1], [C_2]\rangle [C_1]=
$$
\begin{equation}
\label{homology}
\gamma-l\langle [C_4], \gamma \rangle [C_4]-k\langle [C_3], \gamma \rangle [C_3]+kl\langle [C_4], \gamma \rangle \langle [C_3], [C_4]\rangle [C_3].
\end{equation}

Let us first deal with the case $\langle [C_1], [C_2]\rangle =0$. Because $[C_1]$ and $[C_2]$ are primitive homology classes by assumption, this implies $[C_1]=\pm [C_2]$. We claim that then necessarily $\langle [C_3], [C_4]\rangle=0$. To see this, consider \eqref{homology} for $\gamma=[C_1]$
\begin{equation}
\label{deg1}
l \langle [C_4], [C_1]\rangle [C_4]=k(l\langle [C_4], [C_1]\rangle \langle [C_3], [C_4]\rangle - \langle [C_3], [C_1]\rangle)[C_3]. 
\end{equation}
If $\langle [C_4], [C_1]\rangle\neq 0$, then Equation~\ref{deg1} implies that $[C_4]=\lambda [C_3]$ for some $\lambda \in \mathbb{Q}$ so $\langle [C_3], [C_4]\rangle=0$. 

If $\langle [C_4], [C_1]\rangle=0$, then Equation~\eqref{deg1} implies that $\langle [C_3], [C_1]\rangle=0$. Because $[C_1]$, $[C_3]$, and $[C_4]$ are primitive homology classes, we have $[C_3]$, $[C_4]=\pm [C_1]$ so $\langle [C_3], [C_4]=0$. 

Now we assume that $\langle [C_1], [C_2]\rangle \neq 0$. Lemma~\ref{discriminant} shows that
$$
m^2 n^2 \langle [C_1], [C_2] \rangle^4-4mn\langle [C_1], [C_2] \rangle^2=k^2 l^2 \langle [C_3], [C_4] \rangle^4-4kl\langle [C_3], [C_4]\rangle^2. 
$$
By rearranging the terms we get
$$
(mn\langle [C_1], [C_2]\rangle^2-kl \langle [C_3], [C_4]\rangle^2)(mn\langle [C_1], [C_2]\rangle^2+kl\langle [C_3], [C_4]\rangle^2)=
$$
$$
4(mn\langle [C_1], [C_2]\rangle^2-kl \langle [C_3], [C_4]\rangle^2). 
$$
If $mn\langle [C_1], [C_2]\rangle^2-kl \langle [C_3], [C_4]\rangle^2\neq 0$, we can simplify to get
\begin{equation}
\label{contradiction1}
mn\langle [C_1], [C_2]\rangle^2+kl\langle [C_3], [C_4]\rangle^2=4.
\end{equation}
Note that $\langle [C_1], [C_2]\rangle^2\geq 1$ and $\langle [C_3], [C_4]\rangle^2\geq 1$. Because $m+n+k+l=12$, at least one of the summands on the left-hand side of Equation \eqref{contradiction1} is larger than or equal to 5. Therefore, Equation \eqref{contradiction1} cannot hold.   
\end{proof}

\section{Topological uniqueness for Lefschetz fibration of extremal rational type}
\label{section-1119}
In this section we prove the topological uniqueness of Lefschetz fibrations of extremal rational type. To do this, we prove that the intersection numbers of $C_1$, $C_2$, $C_3$ with $C$ are related by a Markov-type equation \eqref{markov}. We then verify that the transitive action of the braid group on the set of positive integral solutions of Equation \ref{markov} is realized at the level of vanishing cycles by Hurwitz moves (Lemma \ref{4steps}). This implies that by a sequence of Hurwitz moves, we can reduce the intersection numbers of $C_1$, $C_2$, $C_3$ with $C$ to~1. After this, the problem is solved by an application of the change of coordinates principle in the sense of Margalit--Farb \cite{Margalit} (Lemma \ref{changecoordinates}). 

Let $C_1$, $C_2$, $C_3$ be non-separating simple closed curves in $\Sigma_{1, 1}$ such that 
\begin{equation}
\label{perfactorization}
\tau_{C_1}^l \tau_{C_2}^m \tau_{C_3}^n =\delta \tau_{C}^{-9}
\end{equation}
for some non-separating simple closed curve $C$. The integers $l$, $m$, $n$ are some permutation of $l_0$, $m_0$, $n_0$ of Table~1.  

Denote the algebraic intersections numbers as follows
\begin{equation}
\label{xyz}
x=\langle [C], [C_1]\rangle, \qquad y=\langle [C], [C_2]\rangle, \qquad z=\langle [C], [C_3] \rangle.
\end{equation}
\begin{lemma}
The following equations hold (possibly after a change of orientation of $C_1$ and $C_3$)
\begin{equation}
\label{orientation}
\langle C_1, C_2 \rangle=-\sqrt{\frac{(12-l-m-n)n}{lm}}z, \qquad \langle C_2, C_3 \rangle=-\sqrt{\frac{(12-l-m-n)l}{mn}}x.
\end{equation}
\end{lemma}
\begin{proof}
Let us apply Proposition~\ref{intersections} to the factorizations
$$
\tau_{C_1}^l\tau_{C_2}^m=\delta \tau_{C}^{l+m+n-12}\tau_{C_3}^{-n}
$$
and 
$$
\tau_{C_2}^m\tau_{C_3}^n=\delta \tau_{C_1}^{-l} \tau_{C}^{l+m+n-12}. 
$$
Then 
$$
lm\langle C_1, C_2 \rangle^2=(12-l-m-n)nxz^2, \qquad mn\langle C_2, C_3\rangle^2=(12-l-m-n)l x^2. 
$$
Because the Dehn twist around a curve does not depend on the orientation of the curve, we may orient $C_1$ and $C_3$ in such a way that \eqref{orientation} holds. 
\end{proof}
\begin{proposition}
The following equation holds 
\begin{equation}
\label{markov}
lx^2+my^2+nz^2=\sqrt{lmn(12-l-m-n)}xyz. 
\end{equation}
\end{proposition}
\begin{proof}
Consider the factorization
\begin{equation}
\label{23split}
\tau_{C_2}^m\tau_{C_3}^n=\delta \tau_{C_1}^{-l} \tau_{C}^{l+m+n-12}. 
\end{equation}
The right-hand side of Equation~\eqref{23split} applied to $[C]\in H_1(\Sigma_{1, 1}, \mathbb{Z})$ gives
\begin{equation}
\label{lhs23}
\delta \tau_{C_1}^{-l} \tau_{C}^{l+m+n-12}[C]=[C]-l\langle [C_1], [C]\rangle [C_1],
\end{equation}
while the left-hand side of Equation~\eqref{23split}
gives 
\begin{equation}
\label{rhs23split}
\tau_{C_2}^m \tau_{C_3}^n[C]=[C]+m\langle [C_3], [C]\rangle [C_3]+n\langle [C_2], [C]\rangle [C_2]+mn\langle [C_3], [C]\rangle \langle [C_2], [C_3]\rangle [C_2]. 
\end{equation}
Considering the intersection number of~\eqref{lhs23} and~\eqref{rhs23split} with $[C]$ we get the following
\begin{equation}
\label{unfin}
-lx^2=mz^2+ny^2+mnz\langle [C_2], [C_3]\rangle y. 
\end{equation}
Equations~\eqref{orientation} and~\eqref{unfin} imply the statement of the proposition. 
\end{proof}
\begin{proposition}
\label{nonzero}
$x, y, z\neq 0$. 
\end{proposition}
\begin{proof}
Assume without loss of generality that $x=0$. Then Equation~\eqref{markov} implies that $y$ and $z$ are zero as well, i.e. $[C_1]$, $[C_2]$, $[C_3]$ all have zero algebraic intersection with $[C]$. Because $[C_1]$, $[C_2]$, $[C_3]$ are primitive homology classes, this means that all of them are equal to $\pm [C]$. Therefore, we have the following identity 
\begin{equation}
\label{12order}
\tau_{C}^{12}=\delta.
\end{equation}
The right-hand side of Equation \eqref{12order} acts on $H_1(\Sigma_{1, 1}, \mathbb{Z})$ trivially while the left-hand side acts non-trivially because $C$ is assumed to be non-separating. This is a contradiction.
\end{proof}
The following lemma is necessary for the proof of Lemma \ref{4steps} as it controls the effect of Hurwitz moves on $x$, $y$, $z$. 
\begin{lemma}
The following equality holds
\begin{equation}
\label{orientation-y}
\langle [C_1], [C_3]\rangle+m\langle [C_1], [C_2]\rangle \langle [C_2], [C_3]\rangle=\sqrt{\frac{(12-l-m-n)m}{ln}}y.
\end{equation}
\end{lemma}
\begin{proof}
Consider the action of both sides of Equation~\eqref{perfactorization} on $[C] \in H_{1}(\Sigma_{1, 1}, \mathbb{Z})$ 
\begin{align}
\label{aligned}
\nonumber [C]&+n\langle [C_3], [C]\rangle  [C_3]+\\ 
\nonumber m\langle [C_2], [C]\rangle[C_2]&+mn\langle [C_3], [C]\rangle \langle [C_2], [C_3]\rangle [C_2]+\\ \nonumber 
l\langle [C_1], [C]\rangle [C_1]&+ln \langle [C_3], [C] \rangle \langle [C_1], [C_3]\rangle [C_1]+\\ 
lm\langle [C_2], [C]\langle [C_1], [C_2]\rangle [C_1]&+lmn \langle [C_3], [C]\rangle \langle [C_2], [C_3] \rangle \langle [C_1], [C_2]\rangle [C_1]=[C].
\end{align}
Consider the intersection number of both sides of Equation~\eqref{aligned} with $[C]$
\begin{align*}
nz^2+my^2+mn(-z)\langle [C_2], [C_3]\rangle (-y)&+lx^2+\\
ln (-z)\langle [C_1], [C_3]\rangle (-x)+lm (-y)\langle [C_1], [C_2]\rangle (-x)&+lmn (-z)\langle [C_2], [C_3]\rangle \langle [C_1], [C_2]\rangle (-x)=0. 
\end{align*}
Substitute Equation~\eqref{orientation}
\begin{align*}
nz^2+my^2-\sqrt{(12-l-m-n)lmn}xyz&+lx^2+lnxz\langle [C_1], [C_3]\rangle -\\
\sqrt{(12-l-m-n)lmn}xyz&+lmn (-z)\langle [C_2], [C_3]\rangle \langle [C_1], [C_2]\rangle (-x)=0. 
\end{align*}
Simplifying using Equation~\eqref{markov} gives Equation~\eqref{orientation-y}.
\end{proof}
Since Dehn twists do not depend on the orientation on the vanishing cycles, there is some ambiguity in the choice of orientation on $C_1$, $C_2$, $C_3$. Let us fix it by adding additional restrictions. 
\begin{definition}
A choice of orientation on $C_1$, $C_2$, $C_3$ is called \textit{admissible} if $x, y, z>0$ and Equation \eqref{orientation} is satisfied. 
\end{definition}
\begin{lemma}
\label{preferred}
There exists a unique admissible choice of orientation on $C_1$, $C_2$, $C_3$. 
\end{lemma}
\begin{proof}
The uniqueness follows from the requirement $x, y, z>0$. Let us prove existence of an admissible choce of orientation. We have already proved that it is possible to orient $C_1$, $C_2$, $C_3$ in such a way that Equation \eqref{orientation} and thus Equation \eqref{markov} hold. Because the left-hand side of Equation \eqref{markov} is strictly positive (Proposition~\ref{nonzero}), we see that either none of $x, y, z$ is negative (in which case we are done) or exactly two are negative ---- in this case we change the orientation of the two corresponding curves. Such a change of orientation preserves Equation \eqref{orientation}; to see this, assume without loss of generality that we change orientation on $C_1$, $C_2$. Then $\langle [C_1], [C_2]\rangle$ and $\langle [C], [C_3]\rangle$ are not changed while $\langle [C_2], [C_3]\rangle$ and $\langle [C], [C_1]\rangle$ both change sign. Therefore, Equation~\eqref{orientation} is preserved. 
\end{proof}
\begin{definition}
\textit{Mutations} are the following 3 transformations of vanishing cycles
\begin{enumerate}
    \item $(C_1, C_2, C_3)\rightarrow (C_1, -\tau^{m}_{C_2}C_3, C_2)$,
    \item $(C_1, C_2, C_3)\rightarrow (-\tau_{C_1}^{l}C_2, C_1, C_3)$,
    \item $(C_1, C_2, C_3)\rightarrow (C_2, -\tau_{C_2}^{-m}C_1, C_3)$.
\end{enumerate}
\end{definition}
\begin{lemma}
\label{4steps}
Mutations of vanishing cycles preserve the admissible orientation and Equations \eqref{perfactorization} and \eqref{orientation} if we change $l$, $m$, $n$ as follows
\begin{enumerate}
    \item $(l', m', n')=(l, n, m)$ for mutation 1; 
    \item $(l', m', n')= (m, l, n)$ for mutation 2; 
    \item $(l', m', n')=(m, l, n)$ for mutation 3. 
\end{enumerate}
\end{lemma}
\begin{proof}
We only write the proof for mutation 1. The proof for other mutations is analogous. 
    
    Equation~\eqref{perfactorization} is preserved because $\tau^l_{C_1}\tau^n_{-\tau^m_{C_2}C_3}\tau^m_{C_2}=\tau^l_{C_1}\tau_{C_2}^m\tau_{C_3}^n \tau_{C_2}^{-m}\tau^m_{C_2}=\tau^l_{C_1}\tau_{C_2}^m\tau_{C_3}^n=\delta \tau_{C}^{-9}$. Equation~\eqref{orientation} is preserved because
    \begin{align*}
    \langle [C'_1], [C'_2]\rangle &=-\langle [C_1], [C_3]\rangle-m\langle [C_2], [C_3]\rangle \langle [C_1]=-\sqrt{\frac{(12-l-m-n)m}{ln}}y, [C_2]\rangle=\\
    \langle [C'_2], [C'_3]\rangle &=\langle [C_2], [C_3]\rangle=-\sqrt{\frac{(12-l-m-n)l}{mn}}x. 
    \end{align*}
    To verify that the admissible orientation is preserved we only need to verify that $y'=-z+m\langle [C_2], [C_3]\rangle y>0$:
    $$
    -z+m\langle [C_2], [C_3]\rangle y=\sqrt{\frac{(12-l-m-n)lm}{n}}xy-z=\frac{lx^2+my^2+nz^2}{\sqrt{n}z}-z>0. 
    $$
\end{proof}
\begin{lemma}
\label{hurwitz-red}
Any 3 vanishing cycles satisfying \eqref{perfactorization} can be transformed by mutations to vanishing cycles $C'_1$, $C'_2$, $C'_3$ such that 
$$
\langle [C], [C'_1]\rangle=\langle [C], [C_{1, min}]\rangle, \qquad \langle [C], [C'_2]\rangle=\langle [C], [C_{2, min}]\rangle, 
$$
\begin{equation}
\label{minsol}
\langle [C], [C'_3]\rangle=\langle [C], [C_{3, min}]\rangle=1. 
\end{equation}
\end{lemma}
\begin{proof}
In \cite{Karpov} there is a table consisting of 14 Markov type equations. By direct comparison we see that the 14 equations of form (11) defined by Table 1 and the equations in \cite{Karpov} are the same. 

In Section 3.6 \cite{Karpov}, there is the definition of mutations $M_x$, $M_y$, $M_x$ of the solutions of Markov equation. One can check directly that the mutations of the solutions of Markov equation defined in \cite{Karpov} and in the mutations of the solutions of Markov equation induced by mutations of vanishing cycles are the same, up to reordering (namely, $M_x$ corresponds to mutation 3, $M_y$ corresponds to mutation 2, $M_z$ corresponds to mutation 3). 

Proposition 3.7(a) \cite{Karpov} asserts that any solution of Markov equation can be mutated into a minimum solution (i.e. a solution minimizing the sum $x+y+z$). Since their argument is that for a non-minimum solution there exists a mutation reducing the value of $x+y+z$, Proposition 3.7 (a) applies in our setting as well (because the existence of a mutation reducing $x+y+z$ does not depend on the way we order $x$, $y$, $z$). 

The table in \cite{Karpov} lists all minimum solutions of every Markov equation. For 2 equations, there are multiple minimum solutions. For each of the two, one can check directly that minimum solutions are related by mutations in our sense. Namely, for equation (4) \cite{Karpov} we can mutate (1, 2, 1) to (2, 1, 1) by applying mutation 1, then mutation 1, and then mutation 2. For equation (8.4) \cite{Karpov} we can mutate (5, 2, 1) to (5, 1, 2) by applying mutation 1, then mutation 2, then mutation 2. 

This means that for any fixed Markov equation, any 2 solutions are related by mutations in our sense. If one additionally notes that for any minimum solution $(x, y, z)$ of Markov equation $z$ equals 1, Lemma 4.7 follows.
\end{proof}
\begin{lemma}
\label{changecoordinates}
Let $C'_1$, $C'_2$, $C'_3$ be simple closed curves in $\Sigma_{1, 1}$ that satisfy the conclusion of Lemma~\ref{hurwitz-red}. Then the factorizations $\tau_{C'_1}^l \tau_{C'_2}^m \tau_{C'_3}^n$ and $\tau_{C_{1, min}}^l\tau_{C_{2, min}}^m \tau_{C_{3, min}}^n$ are globally conjugate.
\end{lemma}
\begin{proof}
Denote the homology class of $C$ as $[u]\in H_1(\Sigma_{1, 1}, \mathbb{Z})$. Since $C$ is a non-separating simple closed curve, there exists a $[v]\in H_1(\Sigma_{1, 1}, \mathbb{Z})$ such that $[u]$ and $[v]$ form a symplectic basis of $H_1(\Sigma_{1, 1}, \mathbb{Z})$. By the assumptions, we have the following
$$
[C'_1]=p_1[v]+q_1[u], \qquad [C'_2]=p_2[v]+q_2[u], \qquad [C'_3]=[v]+q_3[u],
$$
$$
[C_{1, min}]=p_1[v]+q_{1, min}[u], \qquad [C_{2, min}]=p_2[v]+q_{2, min}[u], \qquad [C_{3, min}]=[v]+q_{3, min}[u]
$$
for some integers $p_1$, $p_2$, $q_1$, $q_2$, $q_3$, $q_{1, min}$, $q_{2, min}$, $q_{3, min}$. Equations~\eqref{orientation} and~\eqref{orientation-y} imply
$$
p_1 q_2-p_2 q_1=p_1 q_{1, min}-p_2 q_{2, min}, \qquad p_2 q_3-q_2=p_2 q_{3, min}-q_{2, min}, \qquad p_1 q_3-q_1=p_1 q_{3, min}-q_{1, min},
$$
or, equivalently, 
$$
p_1 (q_{3, min}-q_3)=q_{1, min}-q_1, \qquad p_2(q_{3, min}-q_3)=q_{2, min}-q_2. 
$$
This together with the identity $\tau_{1, 0}\tau_{q, p}\tau_{1, 0}^{-1}=\tau_{q+p, p}$ implies that the global conjugation of $\tau_{C'_1}^l \tau_{C'_2}^m \tau_{C'_3}^n$ by $\tau_C^{q_{3, min}-q_3}$ is equal to $\tau_{C_{1, min}}^l\tau_{C_{2, min}}^m \tau_{C_{3, min}}^n$. 
\end{proof}
\begin{theorem}
For each of the 14 extremal rational types, there exists a unique Lefschetz fibration over the disc (up to topological equivalence). 
\end{theorem}
\begin{proof}
Lemma~\ref{hurwitz-red} shows that any factorization $\tau_{C_1}\tau_{C_2}\tau_{C_3}$ satisfying~\eqref{perfactorization} can be related by mutations to a factorization $\tau_{C'_1}\tau_{C'_2}\tau_{C'_3}$ with intersection numbers of vanishing cycles given by Equation~\eqref{minsol}. Lemma~\ref{changecoordinates} then shows that $\tau_{C'_1}\tau_{C'_2}\tau_{C'_3}$ can be globally conjugated to $\tau_{C_{1, min}}^l\tau_{C_{2, min}}^m \tau_{C_{3, min}}^n$. Because mutations and global conjugation are invertible, this implies that any two factorizations satisfying~\eqref{perfactorization} can be related by a sequence of mutations and a global conjugation. 
\end{proof}

\section{Lefschetz fibrations with 2 type $I_1$ fibers}
\label{Lefschetz2}
The following definition was implicitly introduced in \cite{Auroux}. 
\begin{definition}
\label{good}
Let $[C_1]$, $[C_2]$ be two primitive homology classes in $H_1(\Sigma_{1, 1}, \mathbb{Z})$ with $\langle [C_1], [C_2]\rangle=n>0$. Let $[u]$, $[v]$ be some symplectic basis of $H_{1}(\Sigma_{1, 1}, \mathbb{Z})$ such that
\begin{equation}
\label{kinv}
[C_1]=[u], \qquad [C_2]=n[v]+k[u]
\end{equation}
for some $k \in \mathbb{Z}$. The residue class $k \: (\mathrm{mod}\:n)$ is called the \textit{Auroux invariant} of the pair of homology classes $[C_1]$, $[C_2]$. 
\end{definition}
\begin{remark}
Because we require $[C_2]$ to be a primitive homology class in Definition~\ref{good}, the Auroux invariant is relatively prime to $n$.  
\end{remark}
\begin{lemma}
The value of Auroux invariant does not depend on the choice of symplectic basis in Definition~\ref{good}. 
\end{lemma}
\begin{proof}
Let $[C_1]$, $[C_2]$ be primitive homology classes with intersection number $n>0$ and let $[u]$, $[v]$ be a symplectic basis of $H_1(\Sigma_{1, 1}, \mathbb{Z})$ such that Equation~\eqref{kinv} holds. The basis vector $[u]$ is uniquely determined by $[C_1]$. Therefore, under a change of basis the second basis vector can only change as $[v]\rightarrow [v]+m[u]$ for some integer $m$. Under such change of basis the value of $k$ in Equation~\eqref{kinv} changes as $k\rightarrow k-mn$. Therefore, $k\:(\mathrm{mod}\:n)$ is independent of the choice of basis.  
\end{proof}
Let $n$ be a positive integer. Let $H_n$ be the set of ordered pairs $([C_1], [C_2])$ of primitive homology classes $[C_1]$, $[C_2]\in H_1(\Sigma_{1, 1,}, \mathbb{Z})$ such that $[C_1]\cdot [C_2]=n>0$. The mapping class group $\mathrm{MCG}(\Sigma_{1, 1})$ acts on $H_n$ 
\begin{equation}
\label{action1}
([C_1], [C_2])\rightarrow (\phi[C_1], \phi[C_2]).
\end{equation} 
There is an action of $B_2$ on $H_n$; the action of the generator $\sigma=\sigma_1$ (see Equation~\eqref{artinpresentation}) is given by 
\begin{equation}
\label{action2}
([C_1], [C_2])\rightarrow (-\tau_{C_1}[C_2], [C_1]),
\end{equation}
where $C_1$ is a simple closed curve representing $[C_1]$

Let $F_n$ be the set of monodromy factorizations $\tau_{C_1}\tau_{C_2}$ of length 2 in $\mathrm{MCG}(\Sigma_{1, 1})$ such that $\langle [C_1], [C_2]\rangle=n$. 

There is a bijective map $f:H_n\rightarrow F_n$ which maps a pair of primitive homology classes $([C_1], [C_2])$ to the monodromy factorization $\tau_{C_1}\tau_{C_2}$, where $C_1$, $C_2$ are simple closed curves representing $[C_1]$, $[C_2]$ respectively. Let us show that the bijection $f:H_n\rightarrow F_n$ is equivariant with respect to $\mathrm{MCG}(\Sigma_{1, 1})$ and $B_2$ (see Appendix).
We prove that the map $f$ defined in Section~\ref{Lefschetz2} is equivariant with respect to the action of $\mathrm{MCG}(\Sigma_{1, 1})$ and $B_2$. 
\begin{lemma}
\label{equiv1}
Let $B_2$ act on $H_n$ according to Equation~\eqref{action2} and act on $F_n$ by Hurwitz moves. Then the map $f$ is equivariant with respect to $B_2$. 
\end{lemma}
\begin{proof}
Let $([C_1], [C_2])$ be an element of $H_n$. Then
$$
f(\sigma ([C_1], [C_2]))=f((-\tau_{C_1}[C_2], [C_1]))=\tau_{\tau_{C_1}C_2}\tau_{C_2}=\sigma f(([C_1], [C_2])). 
$$
\end{proof}
\begin{lemma}
\label{equiv2}
Let $\mathrm{MCG}(\Sigma_{1, 1})$ act on $H_n$ according to Equation~\eqref{action1} and act on $F_n$ by global conjugation. Then the map $f$ is equivariant with respect to $\mathrm{MCG}(\Sigma_{1, 1})$. 
\end{lemma}
\begin{proof}
Let $([C_1], [C_2])$ be an element of $H_n$. Then
$$
f(\phi([C_1], [C_2]))=f((\phi[C_1], \phi[C_2]))=\tau_{\phi C_1}\tau_{\phi C_2}=\phi \tau_{C_1}\phi^{-1}\cdot \phi \tau_{C_2}\phi^{-1}=\phi f([C_1], [C_2]). 
$$
\end{proof}
Therefore, we can define the Auroux invariant of an element of $F_n$ as the Auroux invariant of its inverse image in $H_n$. 
\begin{lemma}
\label{hurwitz}
The Auroux invariant of an element of $H_n$ changes as $k\rightarrow -k^{-1}$ under action of $\sigma\in B_2$.
\end{lemma}
\begin{proof}
Let $([C_1], [C_2])$ be an element of $H_n$ with Auroux invariant equal to $k$. Choose a symplectic basis $[u]$, $[v]\in H_{1}(\Sigma_{1, 1}, \mathbb{Z})$ such that 
$$
[C_1]=[u], \qquad [C_2]=n[v]+k[u].
$$
By Equation~\eqref{dehneffect} we have
$$
-\tau_{C_1}[C_2]=-(n[v]+(k+n)[u]). 
$$
Choose a symplectic basis $[u']$, $[v']$ of $H_{1}(\Sigma_{1, 1}, \mathbb{Z})$ such that 
\begin{equation}
\label{newbasis}
-\tau_{C_1}[C_2]=[u'], \qquad [C_1]=n[v']+k'[u']. 
\end{equation}
Equation~\eqref{newbasis} implies 
$$
1\equiv -k' k \: (\mathrm{mod}\:n). 
$$
Therefore, the Auroux invariant of $\sigma([C_1], [C_2])$ is $-k^{-1}$. 
\end{proof}
\begin{lemma}
\label{minusinverse}
The monodromy factorizations $\tau_{C_1}\tau_{C_2}$ and $\tau_{C_3}\tau_{C_4}$ are equivalent under Hurwitz moves and global conjugation if and only if their Auroux invariants $k_1$, $k_2$ satify either $k_1=k_2$ or $k_1=-k_2^{-1}$.
\end{lemma}
\begin{proof}
Assume that the Auroux invariants of $\tau_{C_1}\tau_{C_2}$ and $\tau_{C_3}\tau_{C_4}$ satisfy $k_1=-k_2^{-1}$. After a Hurwitz move applied to $\tau_{C_1}\tau_{C_2}$ Auroux invariants become equal $k_1=k_2=k$, i.e. we have
$$
[C_1]=[u], \qquad [C_2]=n[u]+k[u], 
$$
$$
[C_3]=[u'], \qquad [C_4]=n[v']+k[u']
$$
for some symplectic bases $[u]$, $[v]$ and $[u']$, $[v']$ of $H_1(\Sigma_{1, 1}, \mathbb{Z})$. Because the symplectic representation of $\mathrm{MCG}(\Sigma_{1, 1})$ is surjective, there exists an element $\phi\in \mathrm{MCG}(\Sigma_{1, 1})$ such that $\phi[u]=[u']$ and $\phi[v]=[v']$. Then $\phi$ conjugates $\tau_{C_1}\tau_{C_2}$ to $\tau_{C_3}\tau_{C_4}$. 

To see the converse statement, note that global conjugation does not change the Auroux invariant of a factorization. Therefore Lemma~\ref{hurwitz} implies that if $\tau_{C_1}\tau_{C_2}$ and $\tau_{C_3}\tau_{C_4}$ are equivalent under Hurwitz moves and global conjugation, then either $k_1=k_2$ or $k_1=-k_2^{-1}$. 
\end{proof}
Let $(\mathbb{Z}/m\mathbb{Z})^*$ be the set of invertible elements of the monoid $\mathbb{Z}/m\mathbb{Z}$. 
\begin{lemma}
\label{count}
The set of equivalence classes of the elements of $F_n$ under Hurwitz moves and global conjugation is in bijection with the quotient of $(\mathbb{Z}/m\mathbb{Z})^*$ by the involution $k\rightarrow -k^{-1}$.  
\end{lemma}
\begin{proof}
The invariance of the Auroux invariant under global conjugation and Lemma~\ref{hurwitz} imply that the Auroux invariant map $k:H_n\rightarrow (\mathbb{Z}/m\mathbb{Z})^*$ induces a map from the set of equivalence classes of the elements of $F_n$ under Hurwitz moves and global conjugation to the quotient $(\mathbb{Z}/m\mathbb{Z})^*$ by involution $k\rightarrow -k^{-1}$.  

The induced map is surjective since $k$ is surjective. 

Lemma~\ref{minusinverse} implies that the induced map is injective. 
\end{proof}

Define the function $\psi:\mathbb{Z}\rightarrow \{0, 1\}$ as follows
\[ 
\psi(n)=\left\{
\begin{array}{ll}
      1 \: \mathrm{if}\: n=2^i k\:\mathrm{with} \: k \: \mathrm{odd} \: \mathrm{and} \: 0 \leq i \leq 1\\
      0 \: \mathrm{otherwise}.\\
\end{array} 
\right. 
\]
\begin{theorem}
The number of distinct equivalence classes of elements of $F_n$ under Hurwitz moves and global conjugation is 
$$
\frac{\phi(n)+\psi(n)\prod \left(1+(-1)^{\frac{p_i-1}{2}}\right)}{2}, 
$$
where $\phi$ is the Euler totient function and the product ranges over odd prime number dividing $n$. 
\end{theorem}
\begin{proof}
Lemma~\ref{count} implies that we only have to find the cardinality of the quotient of $(\mathbb{Z}/n\mathbb{Z})^*$ by the involution $k\rightarrow -k^{-1}$. This is done in the appendix. 
\end{proof}
\section{Acknowledgments} 
I would first like to thank my Research Science Institute mentor Barış Kartal. I also would like to thank Prof.~Paul Seidel for suggesting the problem of topological uniqueness and Dr.~Tanya Khovanova and Dr.~John Rickert for useful suggestions. I would like to thank the Research Science Institute, the Center for Excellence in Education,
and the Massachusetts Institute for Technology. Lastly, I would like to acknowledge the financial support of JSC NIS.
\section{Appendix}
Let $n$ be a positive integer and $\mathbb{Z}/n\mathbb{Z}$ the monoid of residue classes $\mathrm{mod}\:n$. Let $(\mathbb{Z}/n\mathbb{Z})^*$ be the set of invertible elements of $\mathbb{Z}/n\mathbb{Z}$. By the definition of Euler's totient function $\phi$, we have $|(\mathbb{Z}/n\mathbb{Z})^*|=\phi(n)$. Here we compute the cardinality of the quotient of $(\mathbb{Z}/n\mathbb{Z})^*$ by the involution $k\rightarrow -k^{-1}$.

The fixed points of the involution $k\rightarrow -k^{-1}$ are residue classes $k$ satisfying $k^2=-1 \: (\mathrm{mod}\:n)$. Denote the number of such residue classes by $r(n)$. 
\begin{lemma}
Let $n$ be a positive integer having the following decomposition into prime powers
$$
n=\prod_{i=1}^{m} p_i^{a_i}.
$$
Then $r(n)=\prod_{i=1}^{m} r(p_i^{a_i})$. 
\end{lemma}
\begin{proof}
Let us construct a bijection between the set $R(n)$ of residue classes in $\mathbb{Z}/n\mathbb{Z}$ satisfying $k^2\equiv -1 \: (\mathrm{mod}\:n)$ and the direct product $\prod_{i=1}^m R(p_i^{a_i})$ of the sets of residue classes satisfying $k^2\equiv -1 \: (\mathrm{mod}\:p_i^{a_i})$ for $1 \leq i \leq m$. 

If we have an integer $k$ such that $k^2+1\equiv 0\: (\mathrm{mod}\:n)$, then obviously $k^2+1\equiv 0\: (\mathrm{mod}\:p_i^{a_i})$ for $1\leq i \leq m$. This defines a map $R(n)\rightarrow \prod_{i=1}^m R(p_i^{a_i})$. 

Now suppose we are given a collection of integer $k_1, \dots, k_m$ such that $k_i^2+1\equiv 0 \: (\mathrm{mod}\:p_i^{a_i})$ for $1 \leq i \leq m$. By the Chinese remainder theorem, there exists an integer $k$ such that $k\equiv k_i \: (\mathrm{mod}\:p_i^{a_i})$ for $1 \leq i \leq m$. Moreover, the integer $k$ is unique $\mathrm{mod}\:n$. Therefore, $k^2+1\equiv k_i^2+1\equiv 0 \:(\mathrm{mod}\:n)$ for $1\leq i \leq m$. This implies that $k^2+1\equiv 0 \: (\mathrm{mod}\:n)$ so we have constructed a map $\prod_{i=1}^m R(p_i^{a_i})\rightarrow R(n)$. Because the integer $k$ is unique $\mathrm{mod}\:n$, this map is inverse to the previously constructed map $R(n)\rightarrow \prod_{i=1}^m R(p_i^{a_i})$.
\end{proof}
\begin{lemma}
Let $p$ be an odd prime number and $a$ be a positive integer. Then $r(p^a)=1+(-1)^{\frac{p-1}{2}}$. 
\end{lemma}
\begin{proof}
If $p\equiv 3 \: (\mathrm{mod}\:4)$ then obviously $r(p^a)=0$. Therefore, we may assume that $p\equiv 1 \: (\mathrm{mod}\:4)$. 

Let us prove that $r(p^a)\leq 2$ for any positive integer $a$. Assume that $r(p^{a_0} )\geq 3$ for some positive integer ${a_0} $. Then there exist 2 integers $x$, $y$ such that $x^2\equiv y^2 \equiv -1 \: (\mathrm{mod}\:p^{a_0} )$, $x \not \equiv y \: (\mathrm{mod}\:p^{a_0} )$ and $x+y \not \equiv 0 \: (\mathrm{mod}\:p^{a_0} )$. Then we have
$$
x^2-y^2\equiv (x-y)(x+y)\equiv 0 \: (\mathrm{mod}\:p^{a_0} ). 
$$
Because $x+y \not \equiv 0 \: (\mathrm{mod}\:p^{a_0} )$ and $x-y \not \equiv 0 \: (\mathrm{mod}\: p^{a_0} )$, there exists an integer $s$ such that $0<s<{a_0} $ and 
$$
x-y=tp^s,
$$
where $t$ is an integer relatively prime to $p$. Therefore,
$$
(y+tp^s)\equiv y^2+t^2 p^{2s}+2tp^s y\equiv -1 \: (\mathrm{mod}\:p^{a_0} ).
$$
If we substract $y^2 \equiv -1 \: (\mathrm{mod}\: p^{a_0} )$ from both sides, we have
$$
t^2 p^{2s}+2tp^s y\equiv 0 \: (\mathrm{mod}\:p^{a_0} ). 
$$
This implies
$$
2tp^s y \equiv 0 \: (\mathrm{mod}\:p^{\mathrm{min}(2s, {a_0} )}).
$$
Because $t$ is relatively prime to $p$, this implies that $y \equiv 0 \: (\mathrm{mod}\:p)$. This is a contradiction because by assumption $-y^2\equiv 1 \: (\mathrm{mod}\:p)$, i.e. $y$ is invertible $\mathrm{mod}\:p$.  

Now we prove that $r(p^a)\geq 2$ for any positive integer $a$. Let us first consider the case $a=1$. We know from Euler's criterion that $r(p)>0$, i.e. there exists an integer $k$ such that $k^2\equiv -1 \: (\mathrm{mod}\:p)$. Note that $k'=-k$ satisfies $k'^2\equiv -1 \: (\mathrm{mod}\:p)$ and that $k' \not \equiv k \: (\mathrm{mod}\:p)$ because $p$ is odd. Therefore, $r(p)\geq 2$. 

Assume that we know that $r(p^a)\geq 2$ for some positive integer $a$. We claim that $r(p^{a+1})\geq 2$. Let $k$ be an integer such that $k^2 \equiv -1 \: (\mathrm{mod}\:p^a)$. Let us find an integer $q$ such that $(p^a q+k)^2\equiv -1 \: (\mathrm{mod}\:p^{a+1})$. We have
$$
(p^aq+k)^2\equiv p^{2a}q^2+k^2+2p^a q k\equiv k^2+2p^a qk\equiv -1 \: (\mathrm{mod}\:p^{a+1}),  
$$
or equivalently
\begin{equation}
\label{paq}
p^{a}q\equiv -2^{-1}(k+k^{-1})\: (\mathrm{mod}\:p^{a+1}),
\end{equation}
where the inverses are taken in the monoid $\mathbb{Z}/p^{a+1}\mathbb{Z}$. Because $k+k^{-1}\equiv 0 \: (\mathrm{mod}\: p^a)$, we can divide both sides of Equation~\eqref{paq} by $p^a$ to get $q$. Because $-(p^aq+k)^2\equiv -1 \: (\mathrm{mod}\: p^{a+1})$ and $-(p^a q+k)\not \equiv p^a q+k \: (\mathrm{mod}\: p^{a+1})$, we have $d(p^a)\geq 2$. 
\end{proof}
Considering quadratic residues $\mathrm{mod}\:4$, we see that $d(2^a)$ is 1 for $0 \leq a \leq 1$ and 0 for $a\geq 2$. 
\begin{proposition}
Let $n$ be a positive integer having the following decomposition into prime powers
$$
n=2^{a_1}\prod_{i=2}^m p_i^{a_i}. 
$$ 
The cardinality of the quotient of $(\mathbb{Z}/n\mathbb{Z})^*$ by the involution $k\rightarrow -k^{-1}$ is 
$$
\frac{\phi(n)+\psi(n)\prod_{i=2}^m \left(1+(-1)^{\frac{p_i-1}{2}}\right)}{2}, 
$$
where \[ 
\psi(n)=\left\{
\begin{array}{ll}
      1 \: \mathrm{if}\: n=2^i k\:\mathrm{with} \: k \: \mathrm{odd} \: \mathrm{and} \: 0 \leq i \leq 1\\
      0 \: \mathrm{otherwise}\\
\end{array} 
\right. 
\]
\end{proposition}
\begin{proof}
We have proved that the number of fixed points of the involution equals $r(n)=\psi(n)\prod_{i=2}^{m}\left(1+(-1)^{\frac{p_i-1}{2}}\right)$. Therefore, the cardinality of the quotient is 
$$
\frac{\phi(n)-r(n)}{2}+r(n)=\frac{\phi(n)+\psi(n)\prod_{i=2}^m \left(1+(-1)^{\frac{p_i-1}{2}}\right)}{2}.
$$
\end{proof}

\end{document}